\documentclass{amsart}
\usepackage{amssymb, mathrsfs}
\usepackage[mathlines]{lineno}

\title{Post-critically finite polynomials in characteristic $p$}
\author{Patrick Ingram}
\address{Department of Mathematics and Statistics, York University}
\email{pingram@yorku.ca}

\newcommand{\QQ}{\mathbb{Q}}

\newcommand{\CC}{\mathbb{C}}
\newcommand{\FF}{\mathbb{F}}

\renewcommand{\AA}{\mathbb{A}}

\newcommand{\Res}{\operatorname{Res}}

\newcommand{\ord}{\operatorname{ord}}

\newcommand{\Poly}[1]{\mathscr{P}_{#1}}
\newcommand{\PCF}{\operatorname{PCF}}
\newcommand{\Null}[1]{\mathscr{N}_{#1}}

\newtheorem{theorem}{Theorem}
\newtheorem{lemma}[theorem]{Lemma}

\theoremstyle{remark}

\theoremstyle{definition}

\newtheorem{question}[theorem]{Question}

\begin{document}
\begin{abstract}
It follows from the compactness of the connectedness locus that the locus of PCF polynomials  in the moduli space of polynomials of degree $d\geq 2$ over $\CC$ is a countable union of zero-dimensional subvarieties. Over fields of characteristic $p$ the analogous claim is certainly false, so it is natural to ask how large the components of the PCF locus can be. We show that the components of the PCF locus in the moduli space of polynomials of degree $d\geq 2$ over an algebraically closed field of characteristic $p>0$ have dimension at most $\lfloor d/p\rfloor$ and that this bound is sharp, in the strong sense that for every $p$ and $d$ there are infinitely many distinct PCF components whose dimension attains this bound (except in the case $p=d=2$).
\end{abstract}
\maketitle

\section{Introduction}\label{sec:intro}

Let $k$ be an algebraically closed field. A polynomial $f(z)\in k[z]$ of degree $d\geq 2$ is \emph{post-critically finite} (PCF) if and only if the support of the \emph{critical locus}, the subscheme of $\AA^1$ defined by $f'=0$, is preperiodic under iteration by $f$. If the derivative $f'$ is not identically zero (for example, if $d$ is not divisible by $\operatorname{char}(k)$) then being PCF is equivalent to requiring that the forward orbit $\{c, f(c), f^2(c), \ldots\}$ is finite for each critical point $f'(c)=0$. If $f'=0$ identically, then $f$ is trivially PCF, although there may be points $f'(c)=0$ with infinite forward orbit.

When $k$ has characteristic zero, it follows from the compactness of the connectedness locus for polynomials over $\CC$ (see~\cite[Corollary~3.7]{BrannerHubbard}) that every algebraic family of PCF polynomials up to affine conjugacy has dimension zero. Equivalently, every PCF map defined over a field of characteristic zero is conjugate to one whose coefficients are algebraic over the prime field $\QQ$.

This fails in characteristic $p>0$. For $p\nmid d>p$, the polynomial $z^d+tz^p\in\FF_p(t)[z]$ has only the critical point $z=0$, which is fixed by $f$, but one can check that $f$ is not conjugate to any polynomial defined over $\overline{\FF_p}$. So in characteristic $p>0$, the PCF locus in the moduli space of polynomials of degree $d\geq 2$ may contain algebraic families of positive dimension, and so the rigidity property familiar over $\CC$ does not hold in general. It is natural to ask, though, if we can say anything at all about the dimensions of algebraic families of PCF maps in positive characteristic. Since the property of being PCF is conjugacy-invariant, it makes sense to restrict to a space of polynomials that parametrizes conjugacy classes in a finite-to-one way.
 We write
\[
\Poly{d}=\left\{f(z)=z^d+a_{d-1}z^{d-1}+\cdots+a_1z\right\}\cong\AA^{d-1},
\]
and $\PCF_d\subseteq\Poly{d}$ for the locus of PCF polynomials. We should emphasize that by a subvariety of the PCF locus, we mean a subvariety $V\subseteq \Poly{d}$ whose generic point is PCF. Since every element of $\Poly{d}(\overline{\FF_p})$ is PCF, it does not make sense to consider merely subvarieties $V\subseteq \Poly{d}$ whose every closed point over some particular algebraically closed field is PCF.

\begin{theorem}\label{thm:main}
Let $k$ be an algebraically closed field of characteristic $p>0$, and let $d\geq 2$. Then $\PCF_d\subseteq\Poly{d}$ is a countable union of subvarieties, each of dimension at most $\lfloor d/p\rfloor$, and this bound is attained by infinitely many components in all cases except $(p, d)=(2, 2)$.
\end{theorem}

The exception in the case $p=d=2$ is because the PCF locus is all of $\Poly{d}\cong \AA^1$ in this case. Note that some previous results in this direction exist. Levy~\cite{LevyRigidity} proved a rigidity theory for tamely ramified PCF maps with a superattracting cycle, which in particular implies Theorem~\ref{thm:main} in the case $2\leq d<p$. This case can also be concluded from~\cite{IngramPCF}. Fraser and the author~\cite{FraserIngram} also demonstrated rigidity for PCF polynomials of the form $f(z)=g(z^d)$, where $\deg(g)<p$. Results of Epstein~\cite{Epstein} show rigidity of PCF polynomials in characteristic zero, but with approaches that are relevant to positive characteristic. We should emphasize that by a subvariety of the PCF locus, we mean an irreducible closed subvariety $V\subseteq \Poly{d}$ whose generic point is PCF, and by a component of the PCF locus we mean a subvariety of the PCF locus that is not properly contained in any other.

Equivalently, we can rephrase Theorem~\ref{thm:main} in terms of individual polynomials.

\begin{theorem}\label{thm:pointwise}
Let $k$ be an algebraically closed field of characteristic $p>0$ and let $f\in k[z]$ be a PCF polynomial of degree $d\geq 2$. Then $f$ is conjugate to a polynomial in $\Poly{d}$, and every conjugate of $f$ in $\Poly{d}$ has coefficients generating an extension of $\FF_p$ of transcendence degree at most $\lfloor d/p\rfloor$. Moreover, there are infinitely many pairwise non-conjugate examples of PCF polynomials $f$ defined over fields of positive characteristic witnessing this bound in all cases except $(p, d)=(2, 2)$.
\end{theorem}

Note that in demonstrating the strong sharpness of the bound, we prove something somewhat stronger than stated above.  As noted above, the map $z^d+tz^p$ is PCF for $p\nmid d>p$, and one can produce infinitely many non-conjugate examples over $\FF_p(t)$ just by replacing $t$ by $t^m$. These maps are non-conjugate, but they are related by automorphisms of the algebraic closure of $\FF_p(t)$, and one might worry that the examples in Theorem~\ref{thm:pointwise} are similarly generated from a single polynomial in each case. In fact, we prove distinctness by generating examples with different critical portraits, which cannot be related either by conjugacy or by field automorphisms.

When a theorem claims that a polynomial has a conjugate with a certain property, it is natural to ask whether one can effectively find this conjugate, and in this case the answer is straightforward. As mentioned above, every polynomial is conjugate to a monic polynomial fixing zero (i.e., a representative in $\Poly{d}$) and it is this conjugate whose coefficients will satisfy the transcendence bound if the polynomial is PCF. Similarly, at least when $p\nmid d$ it is easy to write down a PCF polynomial with $n=\lfloor d/p\rfloor$ algebraically independent coefficients, namely $f(z)=z^r H(z)^p$, where $H(z)$ is a monic polynomial of degree $n$ with algebraically independent coefficients, and $r=d-np$. See also examples constructed by Faber~\cite{Faber} of positive-dimensional families of rational functions with a unique critical point, which are PCF.  It is somewhat more delicate to write down infinitely many examples, and to confirm that they are pairwise non-conjugate.

In certain cases of the proof of Theorem~\ref{thm:pointwise}, we show something slightly more explicit. Specifically, we show that under certain additional hypotheses, the coefficients of any PCF polynomial $f$ are algebraic over a specific subset of at most $\lfloor d/p\rfloor$ of these coefficients. In the general case of Theorem~\ref{thm:pointwise}, though, we cannot give such a clear description of the field generated by the coefficients.
In characteristic $2$ we can make a more detailed analysis that accomplishes this in general.

\begin{theorem}\label{thm:char2}
Let $k$ be an algebraically closed field of characteristic $2$, and let \[f(z)=z^d+a_{d-1}z^{d-1}+\cdots+a_1z\in k[z]\] be a PCF polynomial. Then exactly one of the following holds:
\begin{enumerate}
\item $d$ is odd and $f(z)=zH(z)^2$ for a monic polynomial $H$, and so in particular $f$ is defined over the field $\FF_2(a_1, a_3,  ..., a_{d-2})$;
\item $d$ is odd and $f$ is not of the above form, and every coefficient of $f$ is algebraic over the field
$\FF_2\!\left(a_2, a_4, \ldots, a_{d-1}\right)$; or
\item  $d$ is even and for some $1\leq e<d$ every coefficient of $f$ is algebraic over the field
$\FF_2\!\left(a_2, a_4, \ldots, a_{d-2}, a_e\right)$.
\end{enumerate}
\end{theorem}

The proof of Theorem~\ref{thm:pointwise} splits into two cases, according to whether or not $p\mid d$. Perhaps surprisingly, the case $p\mid d$ is quite straightforward, following from a theorem of Chevalley and some elementary estimates in non-archimedean dynamics. Specifically we note that $\alpha\in L/F$ is algebraic over $F$ if and only if $|\alpha|\leq 1$ for every absolute value on $L$ that restricts to the trivial absolute value on $F$. If $f(z)$ is a polynomial in $z^p$ then the result holds trivially, so let $e$ be the largest index with $a_e\neq 0$ and $p\nmid e$. We show by elementary means that any absolute value on $\FF_p(a_1, ..., a_{d-1})$ such that $|a_e|=1$ and $|a_{jp}|\leq 1$ for all $1\leq j<d/p$ must have $|c|\leq 1$ for each critical point $c$ of $f$, and then $|a_i|\leq 1$ for all $i$. It follows that the coefficients of $f$ are algebraic over $\FF_p(a_p, ..., a_{d-p}, a_e)$ which establishes Theorem~\ref{thm:pointwise} in this case.

The case $p\nmid d$ requires more machinery. First, we consider the \emph{critical nullcone} $\Null{d}\subseteq \Poly{d}$ which is defined to be the locus of polynomials $f\in \Poly{d}$ such that $f(c)=0$ for every critical point $f'(c)=0$. The main result of~\cite{IngramPCF} rests on the observation that $\Null{d}=\{z^d\}$ when $\operatorname{char}(k)=0$, and so when $\operatorname{char}(k)>0$ we start here as well. Concretely, for a particular map $f(z)\in \Null{d}$ we write $f(z)=B(z)H(z)^p$ with $\deg(H)$ maximal. The counting argument on orders of roots that characterizes $\Null{d}$ in the characteristic 0 case fails in positive characteristic because of roots of order divisible by $p$, but understanding the coefficients in $B$ of powers of $z^p$ allows us to avoid this. We essentially show that $f$ can be recovered from the coefficients of $H$ and just this subset of the coefficients of $B$, which is at most $n$ coefficients. This proves Theorem~\ref{thm:pointwise} for polynomials in the nullcone. A standard Gr\"{o}bner deformation argument, similar in spirit to arguments used by Levy~\cite{LevyRigidity}, then shows that any PCF polynomial $f\in \Poly{d}$ can be algebraically deformed to one in $\Null{d}$ in a way that preserves the transcendence degree of the field generated by the coefficients, establishing the claim. Theorem~\ref{thm:main} is essentially a rewording of Theorem~\ref{thm:pointwise}, while Theorem~\ref{thm:char2} comes from a more hands-on analysis in this special case.

Our results concern only polynomials, and very little seems to be known about PCF rational functions in positive characteristic. In characteristic zero, it follows from work of Thurston~\cite{DouadyHubbard} that there are no non-isotrivial families of PCF maps other than the Latt\`{e}s examples) while Benedetto, the author, Jones and Levy~\cite[Corollary 6.3]{BIJL} and Pink~\cite{Pink} prove rigidity of PCF quadratic rational functions in odd positive characteristic.  As mentioned above, examples of Faber~\cite{Faber} demonstrate the existence of high-dimensional families of PCF maps in other cases. Presumably, we would start with the following question.
\begin{question}
Let $\mathcal{M}_d$ be the moduli space of rational functions of degree $d\geq 2$ over $\FF_p$, and let $f\in\mathcal{M}_d$ be a PCF point over some extension of $\FF_p$. How large can the transcendence degree of the field of moduli of $f$ be? What is the largest transcendence degree that occurs for maps $f$ with arbitrarily large critical orbits?
\end{question}
Just as in the characteristic zero case, it may be that there are certain high-dimensional families that are easily explained, or have particular critical dynamics, while general PCF families have lower dimension.

In Section~\ref{sec:escape} we use elementary estimates from non-archimedean dynamics along with a theorem of Chevalley to prove Theorem~\ref{thm:pointwise} in the case $p\mid d$.  Section~\ref{sec:nullcone} delves into the geometry of the nullcone, showing that it has no components of dimension greater than $n$, and then using a Gr\"{o}bner deformation argument to derive the case of Theorem~\ref{thm:pointwise} in which $p\nmid d$. Finally, Section~\ref{sec:max} constructs infinite families of maximal dimension, and Section~\ref{sec:proofs} ties together all of the components in the proofs of the three assertions above.

\subsection*{Formalization} A Lean proof of the main result has been produced by Codex and Claude Code, from a draft of this paper and with guidance from the author. The formalization is available at \texttt{https://github.com/pingram3/pcf-char-p-lean}. It is worth noting what this is and what this is not. The Lean proof formalizes the proofs of Theorems~1, 2, and~3 relative to the Lean kernel and MathLib. In practice, many of the numbered intermediate results have been formalized and can be easily identified in the Lean code. The main difference between the exposition here and the Lean proofs is that we use real-valued valuations in the parts of the argument that rely on escape of orbits relative to non-archimedean absolute values, while the formal version uses Krull valuations (which take values in an ordered abelian group). Concretely, any statement of the form ``There exists an absolute value...'' or logically equivalent to this, which the paper tacitly qualifies over real-valued valuations only, has not been formalized. The author's human-readable proof still stand in those cases, but the Lean proof gets at the main results (as stated) by a slightly different path.

\subsection*{Acknowledgements}
The author's research is supported by grants from NSERC and the Fields Institute.
The research reported in this paper was assisted by the use of various digital tools, including large language models. Discussions with consumer-grade models from OpenAI and Anthropic contributed to the ideas that extended these results from the case $p\leq d< 2p$ to the general case.  Singular and Sage computations were part of the exploratory process that led to the results in this paper, and the author also made use of Google's search engine in locating various resources.



\section{Absolute values and escape}\label{sec:escape}

We begin with the tools and techniques that will allow us to present the proof of Theorem~\ref{thm:pointwise} in the case $p\mid d$. 
We start with a formulation of algebraicity in terms of absolute values on fields, which will allow us to use standard estimates in non-archimedean dynamics. (Note that,  since we are ultimately working in fields of positive characteristic, all absolute values are non-archimedean.)

\begin{lemma}\label{lem:valcrit}
Let $L/K$ be an extension of fields and let $\alpha\in L$. Then $\alpha$ is algebraic over $K$ if and only if $|\alpha|\leq 1$ for every non-archimedean absolute value on $L$ that restricts to the trivial absolute value on $K$.
\end{lemma}

\begin{proof}
One direction is familiar. If $\alpha\neq 0$ is algebraic over $K$, let $P(T)\in K[T]$ be its minimal polynomial. For any non-archimedean absolute value $|\cdot|$ on $L$ restricting trivially to $K$, the assumption $|\alpha|>1$ would give $0=|P(\alpha)|=|\alpha|^{\deg(P)}\neq 0$ by the ultrametric inequality, and hence a contradiction. So we have $|\alpha|\leq 1$ for every non-archimedean absolute value on $L$ that extends the trivial absolute value on $K$.

On the other hand, suppose that $\alpha$ is transcendental over $K$, so that $K(\alpha)\cong K(T)$ is just the field of rational functions in one variable over $K$. The familiar absolute value
\[|g_0(T)/g_1(T)|:=\exp\left(\deg(g_0)-\deg(g_1)\right)\]
for polynomials $g_i$
induces an absolute value on $K(\alpha)$ which restricts trivially to $K$, and has $|\alpha|=e>1$. By a theorem of Chevalley (see for example \cite[Theorem~1, p.~107]{Ribenboim}) this absolute value can be extended to $L$, proving the contrapositive of the other direction of the claim.
\end{proof}

The next two lemmas are the standard escape-rate estimates.

\begin{lemma}\label{lem:escape}
Let $|\cdot|$ be a nonarchimedean absolute value on a field containing the coefficients of
\[
f(z)=z^d+a_{d-1}z^{d-1}+\cdots+a_0,
\]
(where $d\geq 2$) and set
\[
R_f=\max\left\{1,\; |a_j|^{1/(d-j)}: 0\leq j< d\right\}.
\]
If $|\alpha|>R_f$ then $|f^k(\alpha)|=|\alpha|^{d^k}\to\infty$ as $k\to\infty$.
\end{lemma}

\begin{proof}
If $|z|>R_f$ then for each $0\leq j<d$ we have $|z|^{d-j}>|a_j|$, whence $|a_jz^j|<|z|^d$. By the ultrametric inequality $|f(z)|=|z|^d$. Since $|f(z)|=|z|^d>|z|>R_f$, the claim iterates.
\end{proof}

In particular, if $f$ is PCF and $f'$ is not identically zero, then $|f(c)|\leq R_f$ for every critical point $f'(c)=0$. This allows us to confirm Theorem~\ref{thm:pointwise} in the case that $p\mid d$. Note for that derivation that if $f$ is inseparable, then $f(z)=H(z^p)$ for some polynomial $H$ of degree $d/p$, and trivially the coefficients are algebraic (of degree 1) over the field generated by the $d/p$ coefficients of $H$. This extension of $\FF_p$ has transcendence degree at most $d/p$ as it is generated by $d/p$ elements (indeed, transcendence degree at most $d/p-1$ as $H$ is monic with $H(0)=0$).

\begin{lemma}\label{lem:boundPCF}
Let  $k$ be  an algebraically closed field of characteristic $p$, and let \[f(z)=z^d+a_{d-1}z^{d-1}+\cdots+a_1z\in k[z]\] be a PCF polynomial with $p\mid d$. Then either $f$ is inseparable, or else the coefficients of $f$ are algebraic over $F=\FF_p(a_e, a_p, ..., a_{d-p})$,
where $e$ is the maximal index with $p\nmid e$ and $a_e\neq 0$.
\end{lemma}

\begin{proof}
Let $|\cdot|$ be an absolute value of $k$ which restricts trivially to $F$, suppose that $f$ is separable, and write $f(z)=H(z^p)+G(z)$, where $G\neq 0$ contains no monomial of degree divisible by $p$. We have
\[f'(z)=G'(z)=ea_e\prod_{i=1}^{e-1} (z-c_i),\]
with $e=\deg(G)$
and so (comparing coefficients)
\[ja_j = (-1)^{e-j} ea_e \sigma_{e-j}(c_1, ..., c_{e-1}),\]
where $\sigma_n$ is the elementary symmetric polynomial of degree $n$. For $j$ with $p\nmid j$ and $\rho =\max\{1,  |c_1|, ..., |c_{e-1}|\}$, we immediately obtain from $|a_e|=1$ that
\begin{equation}\label{eq:aj}|a_j|\leq \rho ^{e-j}.\end{equation}
If $\rho>1$ and $c$ is a critical point with $|c|=\rho$, then~\eqref{eq:aj} gives $|G(c)|\leq \rho^e$, while the triviality of $|\cdot|$ on $F$ gives
\[|f(c)-G(c)|=|H(c^p)|=|c|^d=\rho^d>\rho^e,\]
and hence $|f(c)|=\rho^d$. But Lemma~\ref{lem:escape} with~\eqref{eq:aj} and the fact that $f$ is PCF gives $|f(c)|\leq \rho$, contradicting $\rho>1$.

It must be that $\rho\leq 1$, and so $|c_i|\leq 1$ for all $i$. But then~\eqref{eq:aj} gives $|a_j|\leq 1$ for all $j$ with $p\nmid j$. Since $|\cdot|$ was an arbitrary absolute value whose restriction to $F$ was trivial, Lemma~\ref{lem:valcrit} allows us to conclude that $a_j$ is algebraic over $F$ for all $j$.
\end{proof}


\section{The critical nullcone}\label{sec:nullcone}

In this section we let $k$ be an algebraically closed field of characteristic $p>0$, and let $d\geq 2$ with $p\nmid d$. Write
\[
d=np+r,\qquad 1\leq r\leq p-1,\qquad n=\lfloor d/p\rfloor.
\]
Let $k/\FF_p$ be any algebraically closed extension.

The \emph{critical nullcone} $\Null{d}\subseteq\Poly{d}$ is the set of $f\in\Poly{d}$ such that $f(c)=0$ for every critical point $f'(c)=0$. Note that if we set
\begin{equation}\label{eq:critcharpoly}
\Phi_f(T):=\Res_z\!\left(d^{-1}f'(z),\;T-f(z)\right)=\prod_{i=1}^{d-1}\left(T-f(c_i)\right),
\end{equation}
then the coefficients of $\Phi_f(T)-T^{d-1}$ are polynomials in the coefficients of $f$, and the vanishing of these coefficients as functions on $\Poly{d}\cong \AA^{d-1}$ cut out precisely $\Null{d}$. A similar construction is used by Looijenga~\cite{Looijenga}. As noted in the introduction, a counting argument shows that $\Null{d}=\{z^d\}$, the origin in $\Poly{d}$, when $k$ has characteristic zero. 

When $\operatorname{char}(k)=p>0$, the nullcone can have components of higher dimension. For example, when $p=3$ and $d=5$ computing the resultant above shows that the nullcone is defined by the vanishing of the four polynomials
\begin{gather*}
 a_4(a_1 + a_2a_4 + a_3a_4^2 - a_4^4)\\
(a_3-a_4^2)(a_2a_4-a_1)^2\\
 a_1^3a_2 - a_1^3a_3a_4 + a_1^2a_2a_3^2 - a_1^2a_2a_3a_4^2  + a_1a_2^2a_3^2a_4 - a_1a_2^2a_3a_4^3  - a_2^4a_4^3 - a_2^3a_3^3 + a_2^3a_3^2a_4^2 \\
 a_1^2(a_1^3 + a_1^2a_3^2 + a_1a_2a_3^2a_4 + a_1a_3^4
         - a_1a_3^3a_4^2 - a_2^2a_3^3 - a_2^3a_4^3 + a_2^2a_3^2a_4^2).
\end{gather*}
One can then calculate that $\Null{d}$ has four irreducible components, each of dimension 1 and corresponding to the families $z^5+t z^2$, $z^5+tz^3$, $z^5+tz^4+t^2z^3$, and $z^5+tz^4-t^3z^2-t^4z$. Note, in relation to the results below, that the coefficients in each family are algebraic over the coefficient of either $z^3$ or $z^2$.

First, we prove an important reduction in the proof of Theorem~\ref{thm:pointwise}, namely that one is free to consider only PCF polynomials in the nullcone. The general approach here is standard (see, for example, \cite[Theorem~15.17]{Eisenbud}) and a similar approach was employed by Levy~\cite{LevyRigidity}.

\begin{lemma}\label{lem:degen}
Suppose that $p\nmid d$. For every PCF polynomial $f\in \Poly{d}$ there is a polynomial  $g\in \Null{d}$ whose coefficients generate a field of the same transcendence degree as those of $f$.
\end{lemma}

\begin{proof}
Now suppose that \[f(z)=z^d+a_{d-1}z^{d-1}+\cdots +a_1 z\] is PCF, and let $L=\FF_p(a_1, ..., a_{d-1}, c_1, ..., c_{d-1})$ where $f'(z)=d\prod_{i=1}^{d-1}(z-c_i)$. Since the $c_i$ are algebraic over the field generated by the coefficients of $f$, the transcendence degree of $L$ is the same as the transcendence degree of the coefficient field.

In $L[t]$, let $A_j=t^{d-j}a_j$ and $C_i=t c_i$ and let \[R=\FF_p[t, A_1, ..., A_{d-1}, C_1, ..., C_{d-1}]\subseteq L[t]\] be the Rees algebra, which is finitely generated as an algebra over $\FF_p$ and has fraction field $L(t)$. Let $\mathfrak{q}\supseteq tR$ be a minimal prime containing $tR$, and let $E$ be the field of fractions of $R/\mathfrak{q}$. Because $t$ is nonzero and not a unit, the principal ideal theorem shows that $\mathfrak{q}$ has height 1. Since $R$ is finitely generated  over $\FF_p$, we have $\dim(R/\mathfrak{q})+\mathrm{ht}(\mathfrak{q})=\dim R$ which is the transcendence degree of $L(t)$ (see \cite[Theorem~A and Corollary~13.4]{Eisenbud})), so $E$ has transcendence degree one less, or in other words the transcendence degree of $L$ over $\FF_p$.  We will show that there is a $g\in \Null{d}$ whose coefficients generate a field over which $E$ is algebraic.

Set \[F(z)=z^d+A_{d-1}z^{d-1}+\cdots +A_1 z\in R[z].\] A direct computation shows that $F'(z)=d\prod_{i=1}^{d-1}(z-C_i)$.

Let $Y_{i, m}=t^{d^m}f^m(c_i)$ for all $m$, including $Y_{i, 0}=C_i$. Note that, since
\begin{equation}\label{eq:Ynot}Y_{i,m+1} = Y_{i,m}^{d} +   \sum_j A_jt^{(d-j)(d^m-1)}Y_{i,m}^{j},\end{equation}
we have by induction $Y_{i, m}\in R$. Also, a calculation gives $Y_{i, 1}=F(C_i)$, and then an induction with~\eqref{eq:Ynot} gives $Y_{i, m}\equiv F(C_i)^{d^{m-1}}\ (\mathrm{mod}\ tR)$ for all $m\geq 1$.

 Since $f$ is PCF, there exists for each $i$ an $m_i>n_i$ such that $f^{m_i}(c_i)=f^{n_i}(c_i)$, and so \[Y_{i, m_i}=t^{d^{m_i}}f^{m_i}(c_i)=t^{d^{m_i}}f^{n_i}(c_i)=t^{d^{m_i}-d^{n_i}}Y_{i, n_i}\in tR\subseteq  \mathfrak{q}.\]
In other words, if $\overline{x}$ denotes the image of $x\in R$ in the residue field $E$, we have $\overline{F}(\overline{C_i})^{d^{m_i-1}}=0$ for all $i$, and hence $\overline{F}(\overline{C_i})=0$. This shows that $g=\overline{F}\in \Null{d}(E)$. But $R$ is generated over $\FF_p$ by $t\in\mathfrak{q}$ and the $A_i$ and $C_j$, so $E$ is generated over $\FF_p$ by the $\overline{A_i}$ and $\overline{C_j}$. Since the $\overline{C_j}$ are algebraic over the $\overline{A_i}$, the $\overline{A_j}$ on their own generate a subfield of $E$ over which $E$ is algebraic.

We have constructed an example defined over some field $E$, but it is worth noting that we could take our $g$ to be defined over the same algebraically closed field $k$ as $f$. Specifically, if $x_1, ..., x_\tau$ is a transcendence basis of $E/\FF_p$, and $y_1, ..., y_\tau\in\FF_p(a_1, ..., a_{d-1})\subseteq k$ are algebraically independent, then the map $x_i\to y_i$ extends to an embedding $\sigma:E\to k$, since $k$ is algebraically closed, and $\sigma(g)\in \Null{d}(k)$.
\end{proof}

Thus to prove Theorem~\ref{thm:pointwise}, it suffices to prove that the coefficients of any $f\in\Null{d}$ generate an extension of $\FF_p$ of transcendence degree at most $n$. We need a few intermediate claims on algebraicity of coefficients of polynomials.

\begin{lemma}\label{lem:abc}
Let $A$, $B$, $C$ be polynomials with coefficients in $k$ such that $AB=C$, and $A$ is monic, and $B$ (equivalently $C$) is nonzero. Then the coefficients of $A$ and $B$ are algebraic over the coefficients of $C$.
\end{lemma}

\begin{proof}
Suppose that $B$ has leading coefficient $B_n$. The roots of $C$ are algebraic over the field generated by the coefficients of $C$, so the roots of $A$ and $B$ are algebraic over this field too. But the coefficients of $A$ are contained in the field generated by the roots, so $A$ is defined over some algebraic extension of the field generated by the coefficients of $C$. The same is true for the coefficients of the monic polynomial $B/B_n$. But since $A$ is monic, $B_n$ is also the leading coefficient of $C$ and hence is in the field generated by the coefficients of $C$, so the coefficients of $B$ are algebraic over that field too.
\end{proof}

The following technical lemma replaces the counting argument on orders of vanishing that establishes the description of $\Null{d}$ in the characteristic zero case.

\begin{lemma}\label{lem:phifinite}
Fix integers $t\geq 1$ and $r_1, ..., r_t\geq 1$, set $R=r_1+\cdots +r_t$, assume $p\nmid R$, and set $\ell=\lfloor R/p\rfloor$. Define polynomials \[b_0, ..., b_{R-1}, c_0, ..., c_{t-2} \in \FF_p[\alpha_1, ..., \alpha_t]\] by equating coefficients of powers of $T$ in
\[B(T):=T^R+b_{R-1}T^{R-1}+\cdots +b_0=\prod_{i=1}^t (T-\alpha_i)^{r_i}\]
and
\[C(T):=RT^{t-1}+c_{t-2}T^{t-2}+\cdots +c_0 =  \sum_{i=1}^{t}r_i\prod_{j\neq i}(T-\alpha_j).\]
Then 
each $\alpha_i$ is integral over $\FF_p[b_0, b_p, ..., b_{\ell p}, c_0, ..., c_{t-2} ]$.
\end{lemma}

\begin{proof}
Note that
\begin{equation}\label{eq:deriv}
B'(T)=C(T)\prod_{i=1}^t(T-\alpha_i)^{r_i-1}
\end{equation}
over $\FF_p[\alpha_1, ..., \alpha_t]$.
We are claiming that $\FF_p[\alpha_1, ..., \alpha_t]$ is finitely generated as a module over the subring $\FF_p[b_0,b_p,\ldots,b_{\ell p},c_0,\ldots,c_{t-2}]$. Since the $b_i$ and $c_j$ are homogeneous polynomials of the $\alpha_n$,  it suffices to show that the only solution to
\[b_{0}=b_p=\cdots = b_{\ell p} = c_0 = \cdots = c_{t-2}=0\] in $\overline{\FF_p}$ is at $\alpha_1=\cdots=\alpha_t=0$.

Suppose that for some value of $\alpha_1, ..., \alpha_t\in \overline{\FF_p}$ we have $b_0=\cdots =b_{\ell p}=0$ and $c_0=\cdots =c_{t-2}=0$, so in particular $C(T)=RT^{t-1}$ (with $R\neq 0$). Write $S=\{i :\alpha_i=0\}$ and set $s=\sum_{i\in S}r_i$. Note that
$\ord_{T=0}(B)=s$ and so in particular $b_s\neq 0$, where $b_R:=1$ in the case that $s=R$. Since $b_{ip}=0$ for all $i$ we have $p\nmid s$ and $s\geq 1$. The product rule for the derivative with~\eqref{eq:deriv} then gives
\begin{equation*}s-1=\ord_{T=0}B'(T)=\ord_{T=0} C + \sum_{i\in S}(r_i-1)=t-1+s-\# S.\end{equation*}
We have $\# S = t$ and so $\alpha_i=0$ for all $i$.

\end{proof}

It will be convenient, for any polynomial $g(z)\in k[z]$, to write
\[g(z)=\sum_{i=0}^{p-1} g^{[i]}(z^p)z^i,\]
so that $g(z)=b_m z^m+\cdots + b_0$ gives
\[g^{[i]}(z^p)z^i=\sum_{j\equiv i\ (\mathrm{mod}\ p)} b_j z^j.\]
Note also that
\begin{equation}\label{eq:wild}
\left(W(z^p)\,g(z)\right)^{[j]}(T)=W(T) g^{[j]} (T)
\end{equation}
for any polynomial $W$.

\begin{lemma}\label{lem:twocases}
Let $f\in \Null{d}$ and assume $p\nmid d$. Then the coefficients of $f$ are algebraic over the coefficients of either $f^{[0]}$ or $f^{[r]}$.
\end{lemma}

\begin{proof}
Now, let $f\in \Null{d}$ over some algebraically closed extension $k/\FF_p$. We may write $f(z)=H(z)^p B(z)$ with
\[H(z)=z^M+h_{M-1}z^{M-1}+\cdots + h_0=\prod_{i=1}^M (z-\gamma_i)\]
and
\[B(z)=z^R+b_{R-1}z^{R-1}+\cdots+ b_0,\]
such that $B$ has no root of multiplicity $p$ or greater. Set $R=\ell p + r$, with $0\leq r < p$.
Note that $R\equiv d\not\equiv 0\ (\mathrm{mod}\ p)$, and in particular $\deg(B)\geq 1$. We will show that the coefficients of $f$ are algebraic over $F=\FF_p(b_0, b_p, ..., b_{\ell p}, \gamma_1, ..., \gamma_M)$. Note that this field is generated by $\ell+1+M=n+1$ elements over $\FF_p$ and so has transcendence degree at most $n+1$. But in fact
\[0=f(0)=(-1)^Mb_0\gamma_1^p\cdots \gamma_M^p\]
implies that one of the claimed generators is 0, and so the transcendence degree of $F/\FF_p$ is at most $n$.

We may write $B(z)=\prod_{i=1}^t(z-\alpha_i)^{r_i}$ with $t\geq 1$, $r_i\geq 1$, and $\alpha_i\in k$ distinct. Write also $C(z)= \sum_{i=1}^{t}r_i\prod_{j\neq i}(z-\alpha_j)$ and note that
\[f'(z)=H(z)^pC(z)\prod_{i=1}^t (z-\alpha_i)^{r_i-1}.\]
Since $f\in \Null{d}$, and since $C$ and $B$ are relatively prime,  the roots of $C$ must be among the roots of $H$. Consequently, the coefficients $c_0, ..., c_{t-2}$ of $C(z)$ are contained in $F$, and so by Lemma~\ref{lem:phifinite} the roots $\alpha_1, ..., \alpha_t$ of $B$ are algebraic over $F$. But the coefficients of $f$ are in the extension of $\FF_p(\alpha_1, ..., \alpha_t, \gamma_1, ..., \gamma_M)$ which we have just shown to be finite over $F$ so our first claim is demonstrated.

Write $H(z)^p=\tilde{H}(z^p)$, so the coefficients of $\tilde{H}$ are $h_0^p, ..., h_M^p$. Note that the roots of $H$ are algebraic over the coefficients of $H$, and hence over the coefficients of $\tilde{H}$, and so $F$ is an algebraic extension of field generated by the coefficients of $\tilde{H}$ and those of $B^{[0]}$. If $B^{[0]}\neq 0$, then $f^{[0]}=\tilde{H}B^{[0]}$ gives us that the coefficients of $\tilde{H}$ and $B^{[0]}$ are algebraic over the coefficients of $f^{[0]}$ by Lemma~\ref{lem:abc}. If, on the other hand, $B^{[0]}=0$, we have just shown that the coefficients of $f$ are algebraic over the field generated by the coefficients of $\tilde{H}$. For any $j$ with $f^{[j]}\neq 0$, the relation $f^{[j]}=\tilde{H}B^{[j]}$ gives that the coefficients of $\tilde{H}$ are algebraic over the coefficients of $f^{[j]}$. Note that $f^{[r]}\neq 0$ as $b_R=1$.
\end{proof}

We are now in a position to establish all but the sharpness claim in Theorem~\ref{thm:pointwise} in the case $p\nmid d$.

\begin{lemma}\label{lem:fofr}
Suppose that $p\nmid d$. The coefficients of any PCF polynomial $f\in \Poly{d}$ generate an extension of $\FF_p$ of transcendence degree at most $n$.
\end{lemma}

\begin{proof}
From Lemma~\ref{lem:degen}, it suffices to prove that the coefficients of any $f\in \Null{d}$ generate an extension of transcendence degree at most $n$. By Lemma~\ref{lem:twocases}, the coefficients of $f$ are algebraic over the coefficients of either $f^{[0]}$ or $f^{[r]}$. Since $f(0)=0$, the polynomial $f^{[0]}$ is a sum of at most $n$ monomials, and so in this case we have the coefficients of $f$ algebraic over an extension of $\FF_p$ by $n$ or fewer coefficients of $f$. The polynomial $f^{[r]}$ is monic of degree $n$, so $f^{[r]}$ has at most $n$ coefficients not already in $\FF_p$. Again, we have the coefficients of $f$ algebraic over a field generated by at most $n$ of these coefficients.
\end{proof}


\section{Maximal families}\label{sec:max}

In this section we prove the following claim, which shows the sharpness of the bound in the main theorem.

\begin{lemma}\label{lem:ptw}
For every $d\geq 2$ with $(p, d)\neq (2, 2)$, there exist infinitely many  PCF $f\in \Poly{d}$, no two related by a conjugacy composed with a field automorphism, such that the coefficients of $f$ generate an extension of $\FF_p$ of transcendence degree  exactly $\lfloor d/p\rfloor$.
\end{lemma}

\begin{proof}
Suppose first that $d=np+r$ with $r\neq 0$ and $n\geq 1$, let $\zeta, u_2, ..., u_n$ be algebraically independent over $\FF_p$, and let $K=\overline{\FF_p(\zeta, u_2, ..., u_n)}$. Over $K[T]$, let
\[F_T(z)=(z-\zeta)^{np}+Tz^{p-r}+\sum_{i=2}^n u_iz^{ip-r},\]
and $f_T(z)=z^rF_T(z)$
so that
\[f'_T(z)=rz^{r-1}(z-\zeta)^{np}.\]
Note that $f_T$ is PCF if and only if the orbit of $\zeta$ is finite, and note also that since $u_2, ..., u_n, (-\zeta)^{np}$ are among the coefficients of $f$, the subfield of $K$ generated by the coefficients after specializing to any $T=u_1$ algebraic over this field still has transcendence degree $n$.

For $P_m(T)=F_T(f^{m-1}_{T}(\zeta))$ we claim that choosing $u_1\in K$ with $P_m(u_1)=0$ will cause $\zeta$ to have exact preperiod $m$ under $f_{u_1}$. Certainly, $P_m(u_1)=0$ implies $f^m_{u_1}(\zeta)=0$, by definition. On the other hand, $f^{m-1}_{u_1}(\zeta)=0$ implies $P_m(u_1)=F_{u_1}(0)=(-\zeta)^{np}\neq 0$, a contradiction, so it must be that $\zeta$ has exact preperiod $m$ under $f_{u_1}$. It remains to show that $P_m(T)$ is nonconstant, so that we may choose such a $u_1$. One can confirm directly that $P_1(T)$ has degree 1 in $T$.  Note that for any non-constant $Q(T)\in K[T]$, we have $\deg_T(F_T(Q))=np\deg_T(Q)$ unless possibly there is cancellation between the terms $Q^{np}$ and $T Q^{p-r}$. For these to be of the same degree, it must be that $((n-1)p+r)\deg(Q)=1$ and hence $\deg(Q)=n=r=1$, so $d=p+1$. If we are not in this case, $P_m(T)$ is non-constant in $T$ for all $m\geq 1$, (indeed $\deg(P_m)=npd^{m-2}$ for $m\geq 2$) so we may choose $u_1\in K$ with $P_m(u_1)=0$. Note that the polynomials $f_{u_1}(z)\in K[z]$ are PCF and pairwise non-conjugate for different $m$, as the preperiods of the critical orbits are invariant under conjugacy and field automorphisms.

The case $\deg(Q)=1$ and $d=p+1$ occurs only once. Here we have $f^0_T(\zeta)=\zeta$, $f_T(\zeta)= T\zeta^{p}+O(1)$, where $O(1)$ is a term in $K$, and so
\[P_2(T)=\zeta^{p^2-p}(\zeta^p+1)T^p-\zeta^p,\]
showing that the polynomial has the expected degree despite the potential for cancellation in the leading coefficient.

In the case $d=np$ with $n\geq 2$, we consider the similar family
\[f_T(z)=z^d+\lambda z^{p-1}(z-\zeta)^{(n-1)p}+Tz^p+\sum_{j=2}^{n-1} u_j z^{pj}\]
with
\[f'(z)=-\lambda z^{p-2}(z-\zeta)^{(n-1)p},\]
so again $f$ is PCF if and only if $\zeta$ is preperiodic for $f$. We take $\lambda, \zeta, u_2, .., u_{n-1}$ to be algebraically independent and solve for $u_1$. Note that the coefficients of $f$ include $\lambda$, $\lambda(-\zeta)^{(n-1)p}, u_2, ..., u_{n-1}$ which are algebraically independent, and specializing at $T=u_1$ algebraic over $\FF_p(\lambda, \zeta, u_2, ..., u_{n-1})$ gives a polynomial whose coefficients generate a field of transcendence degree $n$ over $\FF_p$.

Let
\[F_T(z)=z^{d-p+1}+\lambda (z-\zeta)^{(n-1)p}+T z+\sum_{j=2}^{n-1} u_j z^{pj-p+1},\]
and let $P_m(T)=F_T(f^{m-1}_T(\zeta))$. As above, $f_T(z)=z^{p-1}F_T(z)$, so $\zeta$ has tail length exactly $m$ for $f_{u_1}(z)$ whenever $P_m(u_1)=0$. It suffices to show that the $P_m$ are nonconstant, but we check that for $m\geq 2$ we have
\[\deg(P_m)=(d-p+1)d^{m-2},\]
while $\deg(P_1)=1$.
The roots of the $P_m$ produce infinitely many pairwise non-conjugate PCF polynomials with $n$ algebraically independent coefficients.

We are left with the cases $d\leq p$. For $d<p$ we simply need to exhibit infinitely many pairwise non-conjugate PCF maps, which we can do by considering the family $f_T(z)=z^d+Tz^{d-1}=z^{d-1}(z+T)$. The critical points are $\zeta=-(d-1)T/d$ and, if $d\geq 3$, the fixed point $z=0$. Letting $P_m(T)=f_T^{m-1}(\zeta)+T$, we have $P_1(T)=T/d$ and $\deg(P_m)=d^{m-1}$ for $m\geq 2$, while $\ord_T(P_m)=1$ for all $m\geq 1$. So for each $m\geq 2$ the polynomial $P_m$ has a nonzero root $u_1\in\overline{\FF_p}$. Since $f_{u_1}$ vanishes only at $0$ and $-u_1$, the equation $P_m(u_1)=0$ says that $f_{u_1}^{m-1}(\zeta)=-u_1\neq 0$ and $f_{u_1}^m(\zeta)=0$, and any earlier zero in the orbit of $\zeta$ would force $f_{u_1}^{m-1}(\zeta)=0$. So $\zeta$ has preperiod exactly $m$ under $f_{u_1}$, while the only other critical point, when there is one, is fixed.

For $d=p\neq 2$ let $\lambda$ be transcendental over $\FF_p$ and use $f_T(z)=z^p+\lambda z^2+Tz$. Here $\zeta=-T/(2\lambda)$ is the unique critical point. Write $f_T(z)=zF_T(z)$, so that $F_T(z)=z^{p-1}+\lambda z+T$, and let $P_m(T)=F_T(f^{m-1}_T(\zeta))$. One confirms that $\deg(P_m) = (p-1) p^{m-1}$ and $\ord_T(P_m)=1$ for all $m\geq 1$, so each $P_m$ has a nonzero root $u_1$, algebraic over $\FF_p(\lambda)$. For such a root we have $f^m_{u_1}(\zeta)=0$, while $f^{m-1}_{u_1}(\zeta)=0$ would give $P_m(u_1)=F_{u_1}(0)=u_1\neq 0$. So $\zeta$ has preperiod exactly $m$ under $f_{u_1}$, and the coefficients $\lambda$ and $u_1$ of $f_{u_1}$ generate a field of transcendence degree exactly $1=\lfloor d/p\rfloor$. These specializations again have distinct critical portraits for distinct $m$.

In the case $d=p=2$, the polynomial $f(z)=z^2+tz$ is the generic point of $\Poly{2}$ and is PCF.
\end{proof}


\section{Proofs of the main results}\label{sec:proofs}

We can now pull together the threads needed to establish the results claimed in the introduction. First we make a quick remark about the assertion that the PCF locus in $\Poly{d}$ is a countable union of subvarieties. 
To see this, let \[f(z)=z^d+a_{d-1}z^{d-1}+\cdots +a_1z\in \FF_p[a_1, ..., a_{d-1}, z]\]
be the universal polynomial of degree $d$, and for $m>n\geq 0$ write
\[\Psi_{m,n}=\operatorname{Res}_z(f'(z), T-(f^m(z)-f^n(z)))=\sum_{i=0}^{\deg(f')} \psi_{i, m, n} T^i,\]
where the $\psi_{i, m, n}$ are polynomials in the $a_j$. Let $D=\deg(f')$ be the degree of the derivative of the generic polynomial, which is $d-1$ if $p\nmid d$ and $d-2$ if $p\mid d$. When $p\nmid d$, $\psi_{D, m, n}$ is a unit, and $f$ is PCF if and only if $\Psi_{m, n}=\psi_{D, m, n}T^{D}$ for some $m>n\geq 0$, since this is precisely the condition that every critical point satisfies some non-trivial orbit relation (without loss of generality, there is a single relation satisfied by all critical points). When $p\mid d$, $\psi_{D, m, n}$ need not be a unit, but it is still true that $f$ is PCF if and only if $\Psi_{m, n}$ is a monomial in $T$ for some $m>n\geq 0$, or is 0. Writing $Z_{m n}$ for the closed subset defined by the vanishing of the $\psi_{i, m, n}\psi_{j, m, n}$ for all $0\leq i\neq j\leq D$, the PCF locus is exactly the union of the $Z_{m, n}$ over $m>n\geq 0$.

%
%

\begin{proof}[Proof of Theorem~\ref{thm:pointwise}]
The upper bound in the case $p\nmid d$ is established by Lemma~\ref{lem:fofr}, while the case $p\mid d$ is handled by Lemma~\ref{lem:boundPCF}.  The examples in Section~\ref{sec:max} show the strong sharpness of these bounds. 
\end{proof}

\begin{proof}[Proof of Theorem~\ref{thm:main}]
The equivalence of Theorems~\ref{thm:main} and~\ref{thm:pointwise} is essentially by definition.

 If $V$ is an irreducible component of the PCF locus in $\Poly{d}$ defined over an algebraically closed field $k$, then the function field $k(V)$ has transcendence degree $\dim(V)$ over $k$, and the coordinate functions $a_{d-1}, ..., a_1$ on $V$ give a PCF polynomial $f(z)=z^d+a_{d-1}z^{d-1}+\cdots +a_1z$ over $k(V)$, and so \[\dim(V)=\operatorname{trdeg}_k k(a_1, ..., a_{d-1})\leq \operatorname{trdeg}_{\FF_p} \FF_p(a_1, ..., a_{d-1})\leq n\] by Theorem~\ref{thm:pointwise}.

  On the other hand, let $f(z)\in\Poly{d}$ be a PCF polynomial whose coefficients generate an extension $F$ of $\FF_p$ of transcendence degree $s$. Since $\overline{\FF_p}$ is algebraic over $\FF_p$, the field $F\overline{\FF_p}$ has transcendence degree $s$ over $\overline{\FF_p}$, and it is the function field of an irreducible subvariety $W$ of $\Poly{d}$ over $\overline{\FF_p}$ of dimension $s$, namely the closure of the point whose coordinates are the coefficients of $f$. As $\overline{\FF_p}$ is algebraically closed, $W$ is geometrically irreducible, so $V=W\times_{\overline{\FF_p}}k$ is an irreducible subvariety of $\Poly{d}$ over $k$ of dimension $s$. The coordinates of the generic point of $V$ satisfy exactly the same polynomial relations over $\FF_p$ as the coefficients of $f$, and being PCF is detected by such relations (the equations defining the $Z_{m, m'}$ above), so the generic point of $V$ is PCF and $V$ is contained in the PCF locus.

 Now take, for each $m\geq 2$, the polynomial $f_m$ constructed in the proof of Lemma~\ref{lem:ptw} (excluding the case $d=p=2$), whose coefficients generate a field of transcendence degree $n$ and which has a critical point of exact preperiod $m$, every other critical point being fixed. The construction above gives a subvariety $V_m\subseteq\Poly{d}$ over $k$ of dimension $n$ with PCF generic point. Since the critical orbit relations are also polynomial relations over $\FF_p$ in the coefficients, the generic point of $V_m$ has maximal critical preperiod exactly $m$, and so $V_m\neq V_{m'}$ for $m\neq m'$. Indeed the $V_m$ are pairwise non-conjugate, as conjugacy preserves critical preperiods. Finally, the dimension bound prevents $V_m$ from being properly contained in any larger subvariety of the PCF locus, so each $V_m$ is a component.
\end{proof}
 

\begin{proof}[Proof of Theorem~\ref{thm:char2}] In the case of even $d$, Lemma~\ref{lem:boundPCF} already gives an explicit subset of the coefficients of $f$ over which the others are algebraic, for a PCF polynomial of degree $d$. So we consider just the case of odd $d$. A key relation holds here, namely
\[f(z)-zf'(z)=\sum_{i=1}^n a_{2i}z^{2i}.\]
First, suppose that $a_{2i}=0$ for all $i$. Then $f(z)=zH(z)^2$ for some polynomial $H$, and we are in the first case of the theorem. Suppose otherwise that some $a_{2i}\neq 0$, and let $j$ be the maximal such index. Then for any absolute value $|\cdot|$ on the field generated by the coefficients and critical points of $f$ trivial on $F=\FF_p(a_2, a_4, ..., a_{d-1})$, if we have a $f'(c)=0$ with $|c|>1$ then
\[|f(c)|=\left|\sum_{i=1}^n a_{2i}c^{2i}\right|=|c|^{2j}>|c|.\]
Taking $|c|$ maximal, and arguing as in the escape-rate section, we contradict $f$ being PCF. So it must be that $|c|\leq 1$ for every critical point of $f$, and since $|\cdot|$ was arbitrary we see that the critical points of $f$ are algebraic over $F$. But
\[\prod_{i=1}^{d-1} (z-c_i)=f'(z)=z^{d-1}+a_{d-2}z^{d-3}+\cdots + a_1,\]
and so the odd-index coefficients of $f$ are contained in the field generated by the critical points. We have that the coefficients of $f$ are algebraic over $F$ as claimed.
\end{proof}


\end{document}